\documentclass[11pt]{article}
\usepackage[utf8]{inputenc}
\usepackage[a4paper, margin = 0.5in, bottom = 1in]{geometry}
\usepackage{soul}
\usepackage{authblk}
\usepackage{tikz}
\usetikzlibrary{positioning, calc}
\usepackage[frozencache,cachedir=.]{minted}
\date{}
\usepackage{bm}
\usepackage{multicol}
\usepackage{graphicx}
\usepackage{cite}
\usepackage{skak}
\usepackage{bbm}
\usepackage{lmodern}
\usepackage{amsmath,amssymb,amsthm,comment}
\usepackage{amsfonts}
\usepackage{subcaption}
\usepackage{csquotes}
\usepackage{mathtools}
\usepackage{float}
\usepackage[colorlinks,citecolor=red,urlcolor=blue,linkcolor=blue]{hyperref}
\usepackage{mathrsfs,dsfont}
\usepackage{graphicx}
\usepackage{enumerate}
\usepackage{nicefrac}
\usepackage[font={normalsize},width=1.12\linewidth]{caption}
\usetikzlibrary{arrows}

\newtheorem{nummer}{ }
\newtheorem{thm}[nummer]{Theorem}
\newtheorem{prp}[nummer]{Proposition}
\newtheorem{lem}[nummer]{Lemma}
\newtheorem{cor}[nummer]{Corollary}

\newtheorem{citedthm}{Theorem}

\makeatletter
\def\opargproof[#1]{\par\noindent {\bf #1 }}
 \makeatother

\begin{document}
\medskip\medskip
\begin{center}
\vspace*{50pt}
{\LARGE\bf The Neighbour Sum Problem on Trees}

\bigskip
{\small Sayan Dutta}\\[1.2ex] 
{\scriptsize Département de mathématiques et de statistique, Université de Montréal\\
sayan.dutta@umontreal.ca\\
\href{https://sites.google.com/view/sayan-dutta-homepage}{https://sites.google.com/view/sayan-dutta-homepage}}
\\[1.8ex]

{\small Sohom Gupta}\\[1.2ex] 
{\scriptsize Department of Atmospheric and Oceanic Sciences, McGill University\\
sohom.gupta@mail.mcgill.ca}\\[1.8ex]

\end{center}


\begin{abstract}
  A graph $\mathcal G = (\mathcal V, \mathcal E)$ is said to satisfy the Neighbour Sum Property if there exists some $f:\mathcal V\to\mathbb R$ such that $f\not\equiv 0$ and it maps every vertex to the sum of the values taken by its neighbours. In this article, we provide an algorithm to check whether a given finite tree satisfies the neighbour sum property. We also find a \textit{large} class of trees on $n$ vertices that satisfy the property.
\end{abstract}

\section{Introduction.}
In \cite{sayan}, the present authors along with A. Mandal and S. Chatterjee introduced what they called the \textit{Neighbour Sum Problem}. In this article, we deal with the similar problem on trees. To make this article self contained, we formally pose the problem in the present context. Begin by recalling that a locally finite tree is a connected acyclic undirected graph where every vertex has finite degree.

Let $\mathcal{T} = (\mathcal V, \mathcal E)$ be a locally finite (unlabelled) tree. Call a function $f_{\mathcal T} : \mathcal V \to \mathbb R$ to satisfy the \textit{neighbor sum property} if
\[f_{\mathcal T}(x) = \sum_{\{x,y\}\in\mathcal E}f_{\mathcal T}(y)\]
for all $x\in \mathcal V$.

We are interested in classifying all pairs $(\mathcal T, f_{\mathcal{T}})$ satisfying the neighbor sum property such that $f_{\mathcal T}\not\equiv 0$. Such an $f_{\mathcal T}$ will be called a solution for the neighbour sum problem, and $\mathcal{T}$ will be called to satisfy the neighbour sum property if it admits such an $f_{\mathcal T}$. For clarity, we will drop the subscript of $f$ unless in conflict with other notation.

Section \ref{Not} introduces some important notation to set up the algorithm presented in Section \ref{app}. Section \ref{sol} proves the condition(s) necessary for a finite tree to satisfy the neighbour sum property, and Section \ref{count} deals with computing the number of trees on $n$ vertices that do so. Section \ref{infinite} goes a step further and presents necessary condition(s) for locally finite infinite trees to satisfy the property. Finally, Section \ref{rem} presents some simple results for other kinds of graphs, which may provide a direction to extend our analysis to general graphs.

\section{Notation}\label{Not}
Set a vertex $v_0\in \mathcal{T}$ and call it the \textit{root vertex}. All objects in this section, unless specified otherwise, are defined with respect to the root vertex $v_0$. We shall not use $v_0$ in the definitions to avoid being cumbersome. 

Define the set $A_1=\{v_1: \{v_1, v_0\}\in \mathcal{E}\}$. Call this set the set of nearest-neighbours of $v_0$. Now recursively define
$$A_j \coloneqq \{v\notin A_{j-2} : \exists x \in A_{j-1} : \{v,x\}\in \mathcal E\}$$
to be the set of $j$-nearest neighbours of $v_0$. Note that any $v\in A_j$ is connected to $v_0$ by unique elements\footnote{If not, then we get at least two different paths from $v$ to $v_0$, which creates a cycle, contradicting that $\mathcal T$ is a tree.} in $A_{j-1},\:A_{j-2},\ldots,\:A_1$. This sets up the following indexing scheme (denote $N(v)$ to be the neighbour set of $v$):\begin{itemize}

    \item Index $A_1$ as $\{v_1(1),\: v_1(2),\ldots,\:v_1(|A_1|)\}$
    \item Index $A_2\cap N(v_1(r_1))$ as $\{v_2(r_1, 1),\: v_2(r_1, 2),\ldots\}$
    \item In general, index $A_{k+1}\cap N(v_k(r_1,\: r_2,\ldots,\: r_k))$ as $\{v_{k+1}(r_1,\:r_2,\ldots,\:r_k,1), v_{k+1}(r_1,\:r_2,\ldots,\:r_k,2),\ldots\}$
\end{itemize}

Notice that
$$A_{j+1}=\bigsqcup_{r_1,\:r_2,\ldots,\:r_j} A_{j+1}\cap N(v_j(r_1,\: r_2,\ldots,\: r_j))$$
and hence the indexing labels all vertices of the tree uniquely.

For example, a vertex $v_l(r_1,\:r_2,\ldots,\:r_l)$ would be connected to $v_0$ using the path $v_0\leftrightarrow v_1(r_1) \leftrightarrow v_2(r_1, r_2)\ldots\leftrightarrow v_l(r_1, r_2, \ldots, r_l)$. For notational clarity we will use $\bar{r}_l$ to denote a general indexing $r_1,\: r_2,\: \ldots,\:r_l$. A diagram is presented below to make this indexing scheme clearer.

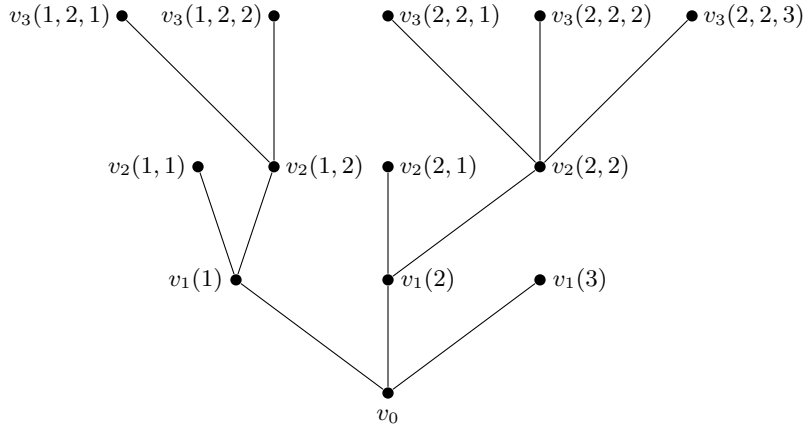
\begin{figure}[htp!]
    \centering
    \begin{tikzpicture}[
    every node/.style={circle, fill=black, inner sep=1.5pt},
    every label/.style={fill=none, inner sep=1pt, font=\footnotesize}
    ]
    
    \node[label=below:{$v_0$}] (v0) at (0,0) {};
    \node[label=left:{$v_1(1)$}] (v11) at (-2,1.5) {};
    \node[label=right:{$v_1(2)$}] (v12) at (0,1.5) {};
    \node[label=right:{$v_1(3)$}] (v13) at (2,1.5) {};
    \node[label=left:{$v_2(1,1)$}] (v211) at (-2.5,3) {};
    \node[label=right:{$v_2(1,2)$}] (v212) at (-1.5,3) {};
    \node[label=right:{$v_2(2,1)$}] (v221) at (0,3) {};
    \node[label=right:{$v_2(2,2)$}] (v222) at (2,3) {};
    \node[label=left:{$v_3(1,2,1)$}] (v3121) at (-3.5,5) {};
    \node[label=left:{$v_3(1,2,2)$}] (v3122) at (-1.5,5) {};
    \node[label=right:{$v_3(2,2,1)$}] (v3221) at (0,5) {};
    \node[label=right:{$v_3(2,2,2)$}] (v3222) at (2,5) {};
    \node[label=right:{$v_3(2,2,3)$}] (v3223) at (4,5) {};
    
    \draw (v0) -- (v11);
    \draw (v0) -- (v12);
    \draw (v0) -- (v13);
    \draw (v11) -- (v211);
    \draw (v11) -- (v212);
    \draw (v12) -- (v221);
    \draw (v12) -- (v222);
    \draw (v212) -- (v3121);
    \draw (v212) -- (v3122);
    \draw (v222) -- (v3221);
    \draw (v222) -- (v3222);
    \draw (v222) -- (v3223);
    
    \end{tikzpicture}
    \caption{Indexing scheme for a finite tree on 13 vertices.}
    \label{fintree}
\end{figure}

Finally, for a \textit{finite} tree with a root vertex $v_0$ such that the longest path starting at $v_0$ is of length $k$, we also define
$$S_{1,k}= S_{1,k}(v_0) \coloneqq \sum_{r_1}\left(1-\sum_{r_{2}}\left(1-\ldots\sum_{r_{k-1}}\left(1-\sum_{r_k} 1\right)^{-1}\ldots\right)^{-1}\right)^{-1}$$
and generally
$$S_{j,k}(\bar{r}_{j-1}) = S_{j,k}(\bar{r}_{j-1}, v_0) \coloneqq \sum_{r_j}\left(1-\sum_{r_{j+1}}\left(1-\ldots\sum_{r_{k-1}}\left(1-\sum_{r_k} 1\right)^{-1}\ldots\right)^{-1}\right)^{-1}$$
for $1<j\leq k$.

This quantity is inherently a structural property of the tree and does not depend upon the values any of the vertices take under some $f$. Note that $\sum_{r_k} 1$ just counts the number of children of a vertex $v_{k-1}(\bar{r}_{k-1})\in A_{k-1}$. Then, merely for the sake of definition consider the quantity $C_{k-1}=(1-\sum_{r_k} 1)^{-1}$. This is an algebraic expression that is then summed over $r_{k-1}$, which means a sum over $C_{k-1}$ over all children of some $v_{k-2}(\bar{r}_{k-2})\in A_{k-2}$. We continue this recursively until we reach the root vertex. Essentially, at each step, we are summing over the children of a vertex in $A_j$, but instead of counting the number of children, we count $C_{j+1}$.

Additionally, we shall make use of the following definitions as formal calculations to substantiate our proofs:\begin{enumerate}
    \item $\dfrac{1}{0} = \infty$
    \item $\dfrac{r}{\pm\infty} = 0$ for any $r\in\mathbb{R}$
    \item $r\pm\infty = \pm\infty$ for any $r\in\mathbb{R}$
    \item $\dfrac{0}{0}$ is finite.
\end{enumerate}

\section{Solutions for finite trees}\label{sol}
\begin{thm}\label{main}
    For a finite tree $\mathcal{T}=(\mathcal{V},\:\mathcal{E})$, there exists a pair $(\mathcal{T}, f)$ satisfying the neighbour sum property if and only if there is at least one vertex $v_0\in \mathcal{V}$ such that $S_{1k}(v_0)=1$.
\end{thm}

\begin{proof}
    First, let us assume that there is a solution to the neighbour sum property and call it $f$. We get the corresponding equations for each $v\in \mathcal{V}$. From now on, we will interpret $v$ not just as the vertex, but also as the value assigned to that vertex under the neighbour sum function - in other words, we will denote $f(v)$ as just $v$ wherever appropriate for a neater presentation.
    
    We have $$v_0 = \sum_{r_1} v_1(r_1)$$ for the root vertex. In general, for every vertex in $A_j$, we can write $$v_j(\bar{r}_{j-1}, r_j) = v_{j-1}(\bar{r}_{j-1})+\sum_{r_{j+1}}v_{j+1}(\bar{r}_{j-1}, r_j, r_{j+1})$$ for all $r_j$. 
    
    Suppose the longest path starting at $v_0$ is of length $k$. Equivalently, the nearest neighbours can be indexed up to $k$ with the sets $A_1,\:A_2,...,\:A_k$. For $A_k$, the neighbour sum property becomes $$v_k(\bar{r}_{k-1}, r_k) = v_{k-1}(\bar{r}_{k-1})$$ for all $r_k$.

    Using the equation for a vertex in $A_{k-1}$, we obtain 
    $$v_{k-1}(\bar{r}_{k-2}, r_{k-1}) = v_{k-2}(\bar{r}_{k-2})\left(1-\sum_{r_k} 1\right)^{-1}$$
    and recursively, 
    \begin{align*}
        & v_{k-2}(\bar{r}_{k-3}, r_{k-2}) = v_{k-3}(\bar{r}_{k-3})\left(1-\sum_{r_{k-1}}\left(1-\sum_{r_k} 1\right)^{-1}\right)^{-1} = v_{k-3}(\bar{r}_{k-3})(1-S_{k-1,k}(\bar{r}_{k-2}))^{-1}\\
        &\implies v_{k-3}(\bar{r}_{k-4}, r_{k-3}) = v_{k-4}(\bar{r}_{k-4})\left(1-\sum_{r_{k-2}}\left(1-\sum_{r_{k-1}}\left(1-\sum_{r_k} 1\right)^{-1}\right)^{-1}\right)^{-1} = v_{k-4}(\bar{r}_{k-4})(1-S_{k-2,k}(\bar{r}_{k-3}))^{-1}\\
        &.\\
        &.\\
        &.\\
        &\implies v_1(r_1) = v_0\left(1-\sum_{r_2}\left(\ldots\left(1-\sum_{r_{k-1}}\left(1-\sum_{r_k} 1\right)^{-1}\right)^{-1}\ldots\right)^{-1}\right)^{-1} = v_0(1-S_{2,k}(r_1))^{-1}\\
        &\implies v_0 = \sum_{r_1}v_1(r_1) = v_0\sum_{r_1}\left(1-\sum_{r_2}\left(\ldots\left(1-\sum_{r_{k-1}}\left(1-\sum_{r_k} 1\right)^{-1}\right)^{-1}\ldots\right)^{-1}\right)^{-1} = v_0S_{1,k}
    \end{align*}

    and hence $S_{1,k} = 1$ if $v_0\neq 0$. Since we assumed $f\not\equiv 0$, there exists some $v_0\neq 0$ implying $S_{1,k}(v_0)=1$.

    To show the converse,  we will explicitly construct the function $f$. To do so, let us set some finite $v_0 = a\neq 0$. Assuming the longest path starting from $v_0$ is of length $k$, we can check that the construction $$v_j(\bar{r}_{j-1},r_j)=v_{j-1}(\bar{r}_{j-1})(1-S_{j+1,k}(\bar{r}_j))^{-1}$$ allows $(\mathcal T, f)$ to satisfy the neighbour sum property at all vertices. 
    
    The only thing that remains to be shown that none of these values become infinitely large. This is only possible if for some $j$, $S_{j+1,k}(\bar{r}_j)=1$ and $v_{j-1}(\bar{r}_{j-1})$ is not zero.

    First we show that all $v_1(r_1)$ are finite. We have
    $$v_1(r_1)=v_0(1-S_{2,k}(r_1))^{-1} = a(1-S_{2,k}(r_1))^{-1}$$
    and hence we must have $S_{2,k}(r_1)=1$ for $v_1(r_1)$ to be infinite. However
    $$S_{1,k} = \sum_{r_1}(1-S_{2,k}(r_1))^{-1} = 1$$
    and so neither of the terms in the sum can be infinite. Therefore, $v_1(r_1)$ is finite for every $r_1$.

    Let us assume the induction hypothesis that along with $v_0$, all elements in $\bigsqcup_{i=1}^j A_i$ are finite.
    
    Check that
    $$S_{j+2,k}(\bar{r}_{j+1}) = \sum_{r_{j+2}}\left(1-S_{j+3,k}(\bar{r}_{j+2})\right)^{-1}$$
    and since we have formally defined $1/0=\infty$, therefore if any one of the terms in the sum becomes infinite, the entire sum becomes $-\infty$. Suppose $S_{j+3,k}(\bar{r}_j)=1$ for some $\bar{r}_j$. This gives us $$
    v_{j+1}(\bar{r}_{j+1}) = v_{j}(\bar{r}_{j})(1-S_{j+2,k}(\bar{r}_{j+1}))^{-1} = \dfrac{v_{j}(\bar{r}_{j})}{-\infty}=0$$
    since $v_{j}(\bar{r}_j)\in A_j$ and so is finite.

    Therefore, by strong induction, $f$ is finite at all vertices and so if there exists some $v_0\in\mathcal{V}$ such that $S_{1,k}(v_0)=1$, then there exists a pair $(\mathcal{T}, f)$ that satisfies the neighbour sum property. This completes the proof.
\end{proof}

Notice that this theorem gives us an algorithm to check whether a given tree satisfies the neighbour sum property. One simply needs to check whether there is a vertex $v_0$ such that $S_{1k}(v_0)=1$. In the next section we will share an algorithm as a python code to count the number of non-isomorphic trees on $n$ vertices that satisfy the property, but before that we wish to discuss applications of Theorem \ref{main}.

\begin{cor}\label{path}
    A path graph $\mathcal{P}$ on $n$ vertices satisfies the neighbour sum property if and only if $n = 3m+2$ for some $m\in \mathbb N$.
\end{cor}
\begin{proof}
    Suppose $\mathcal{P}$ is a path graph on $n=3m+2$ vertices. Denote them in order as $\{v_0, v_1, \ldots, v_{n-1}\}$. Starting from $v_0$ (here $k=n-1$), we see that
    $$S_{1,k}=\cfrac{1}{1 -\cfrac{1}{1 -\cfrac{1}{\ddots -\cfrac{1}{1-1}}}}$$
    where there are $k-1=3m$ fractions. Note that $x=1-S_{1,k}$ is a finite continued fraction with convergents $$x_j=\dfrac{A_j}{B_j}$$ where $A_j=A_{j-1}-A_{j-2}$ and $B_j=B_{j-1}-B_{j-2}$ with the constraints $A_{-1}=1, A_0=1$, $B_{-1}=0, B_0=1$. It is easy to verify the periodic values taken up by both sequences $\{A_j\}_{j=0}^{k} = \{1,0,-1,-1,0,1,1,0,\ldots\}$ and $\{B_j\}_{j=0}^{k} = \{1,1,0,-1,-1,0,1,1,\ldots\}$. For $k=3m+1$, $A_k=0$ and $B_k\neq 0$ giving $x_k=0$ or $S_{1,k}=1$. Note that for $k=3m, 3m+2$, $x_k\neq 0$ so $S_{1,k}\neq 1$.

    Conversely, let there be a path graph on $n$ vertices $\mathcal{P}=(\mathcal{V}, \mathcal{E})$ that satisfies the neighbour sum property, i.e. there exists $f:\mathcal{V}\to\mathbb R$ and $v_0\in\mathcal V$ such that $S_{1k}(v_0)=1$. However, we are only interested in $S_{1k}(v_0)$ if $v_0$ is a leaf, so we will label one of the leaves as $v_0$ and label others through our indexing convention. If $S_{1k}(v_0)=1$, then we know that $n$ must be of the form $3m+2$. If $S_{1k}\neq 1$, then $f(v_0)=0$. So, $f(v_1)=0$. Recursively, by using the neighbour sum property, we can see that $f(v_j)=0$ for all $0\leq j\leq n-1$, which is a contradiction to $f\not\equiv 0$.
\end{proof}

\begin{cor}\label{star}
    A star graph $\mathcal S$ satisfies the neighbour sum property if and only if it has exactly one edge.
\end{cor}
\begin{proof}
    Suppose a star graph has a central vertex $v_0$ and a set of leaves $A_1 = \{v_1(1),\:v_1(2),\:,\ldots,v_1(n)\}$ where $n\geq 1$. Then
    $$S_{1,k}(v_0) = S_{1,1}(v_0)=\sum_{r_1}1 = n$$
    which resolves the $n=1$ case. If $n\neq 1$, then $S_{1,k}(v_0)\neq 1$ and so for all $f:\mathcal{V}\to\mathbb{R}$ satisfying the neighbour sum property, $f(v_0)=0$.
    Now, the neighbour sum property on each $v_1(j)$ gives $v_1(j)=v_0=0$, which gives $f\equiv 0$, which is a contradiction. Thus, the only star graph that satisfies the neighbour sum property is the one with exactly one leaf, i.e. the path graph on two vertices.
\end{proof}

Of course, Corollary \ref{path} and \ref{star} are also easy to verify without using Theorem \ref{main}. So we will now provide a more interesting example. We call a tree $\mathcal L_{d,k}$ to be a \textit{level-symmetric tree} if there exists a (root) vertex $v_0$ such that the following holds simultaneously:- \begin{enumerate}
\vspace{-0.19 in}
    \item $v_0$ has degree $d\ge 1$ and every other vertex has degree $d+1$ or $1$.
    \vspace{-0.12 in}\item every path starting at $v_0$ and ending at a leaf has length $k$.
\end{enumerate}
It should be noted that while the choice of $v_0$ seems convenient, it is indeed an unique choice if $d>1$ as it is the only vertex with degree $d$. If $d=1$, there are two choices of $v_0$, namely the two leaves of the resulting path graph.

\begin{prp}
   A level symmetric tree $\mathcal L_{d,k}$ satisfies the neighbour sum property if and only if $d=1$ and $k\equiv 1(\operatorname{mod}\,3)$.
\end{prp}
\begin{proof}
    Starting from $v_0$, it is easy to see that every vertex at a given level $A_j$ looks the same structurally. Therefore, each of the sums $\sum_{r_j}$ in $S_{1,k}(v_0)$ can be replaced by the summand scaled by the number of children at $r_j(\bar{r}_{j-1})$, i.e. $d$. We obtain
    $$S_{1,k}(v_0) = d\left(1-d\left(1-\ldots d\left(1-d\right)^{-1}\ldots\right)^{-1}\right)^{-1}$$

    We are left with finding out for what values of $d$, $S_{1,k}=1$ for positive integer value(s) of $k$. Check that $1-S_{1,k}$ is a continued fraction with convergents $x_j = a_j/b_j$ where $a_j=a_{j-1}-da_{j-2}$ and $b_j=b_{j-1}-db_{j-2}$ with $a_{-1}=1,\:a_0=1$. $b_{-1}=0, b_0=1$. Since we require $1-S_{1,k}=0$ it is enough to look for $a_k=0$. 

    Such a Fibonacci-like recurrence can be solved using several techniques. Using those or otherwise, it is easy to check
    $$a_k = 0\;\implies\left(\dfrac{1-2d}{2d}+i\dfrac{\sqrt{4d-1}}{2d}\right)^{k+2} = 1$$
    which is only possible if the term in the parentheses is a root of unity. Let us call it $e^{2ir\pi}$ where $r\in\mathbb{Q}$. Then $\cos(2r\pi) = \frac{1}{2d}-1\in\mathbb{Q}$. But, it is known that the only rational multiples of $\pi$ in $[0,\pi)$ whose cosine is also rational are $\pi/3,\:2\pi/3$.\,\footnote{Indeed, $2 \cos (r \pi) = e^{r i \pi} + e^{-r i \pi}$ is an algebraic integer. So, it is a rational number if and only if it is an integer. Also, notice that considering $[0,\pi)$ is enough since cosine is periodic and $\cos(\pi+x)=\cos x$.}
    
    Solving for $d$ gives us the only positive integer solution $d=1$. The condition on $k$ now follows from Corollary \ref{path}.
\end{proof}

Once we have come to a necessary and sufficient condition for solutions to exist on trees, it is natural to ask how many linearly independent solutions can a tree accommodate. We will address this question by proving that the dimension of the solution space can indeed be arbitrarily large.

\begin{prp}
    For each $d\in \mathbb N$, there are trees satisfying the neighbour sum property for which the solution space has dimension $d$.
\end{prp}
\begin{proof}
    Consider a tree that has a central node and $d+1$ path graphs of length two connected to this central node. Define a function $f$ that maps the central node to $0$ and each path graph to a corresponding fixed real $x_1,\:x_2,\ldots,\:x_{d+1}$ in any order (see Figure \ref{central}). Note that by definition all points in each path graph satisfy the neighbour sum property. For the central node to satisfy the property, we must have
    $$x_1+x_2+\ldots+x_{d+1}=0$$
    and this is the only independent equation. So, the tree has $d$ linearly independent solutions in $\mathbb{R}^{d+1}$. Therefore, the dimension of the solution space is $d$.
\end{proof}

\begin{figure}[htp!]
    \centering
    \begin{tikzpicture}[scale = 1.5, every node/.style={circle, draw, fill=black, minimum size=4pt, inner sep=1pt}]

    \def\k{6} 
    \def\rstep{1} 
    
    \node[anchor= north west, label={5:$0$}] (C) at (0,0) {};
    
    \foreach \i in {1,...,\k} {
        \pgfmathsetmacro{\angle}{360/\k * (\i - 1)}
    
        \node[anchor = center, label={[label distance=0cm]{\angle-60}:\small$x_{\i}$}] (A\i) at ($(C)+({\angle}:1)$) {};
        \node[anchor=center, label={[label distance=0cm]{\angle-60}:\small$x_{\i}$}] (B\i) at ($(C)+({\angle}:2)$) {};
    
        \draw (C) -- (A\i) -- (B\i);
    
    }
    
    \end{tikzpicture}
        \caption{A tree with a central node and $n=6$ path graphs of length two connected to it.}
    \label{central}
\end{figure}
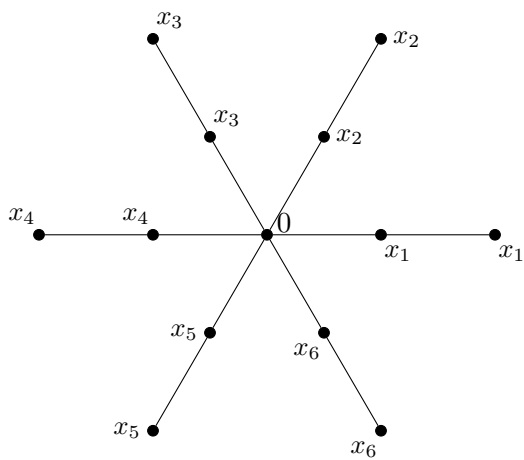

\section{Counting Number of Trees}\label{count}
In this section, we focus on the question of the number of trees that satisfy the neighbour sum property. We let $\tau(n)$ and $\sigma(n)$ denote respectively the number of unlabelled trees on $n$ vertices and the number of those that satisfies the neighbour sum property. The main ingredient we use is the classic theorem stated below.
\begin{citedthm}[R. Otter \cite{otter}]\label{otter}
    We have
    $$\tau(n) \sim C \cdot \alpha^n \cdot n^{-5/2}$$
    for
    constants $C\approx 0.53494$ and $\alpha \approx 2.9557$.\,\footnote{More precisely, $1/\alpha$ is the radius of convergence of $H(z)$ satisfying
    $$H(z) = z\cdot \exp \left( H(z) + \frac{H(z^2)}{2} + \frac{H(z^3)}{3} + \dots \right)$$
    and
    $$C=\frac{3\delta_3}{4\sqrt \pi}$$
    where $\delta_3$ is the coefficient of $Z^{3/2}$ in the expansion of
    $H(z)-\frac 12 H(z)^2$ where $Z=1-\alpha z$.
    } 
\end{citedthm}

For a detailed exposition of this theorem and related analogues, the reader can check out \cite{analcomb}.

Using this, we will prove the following result.
\begin{thm}
    The set of unlabelled trees on $n$ vertices has a positive density subset all of whose elements satisfy the neighbour sum property.
\end{thm}
\begin{proof}
    For a large $n$, take $5$ vertices, arrange them as a line graph and label them $v_1$ through $v_5$ in order. Now, to complete the construction of the graph, create a tree using $v_3$ and all the other vertices except $v_i$, $i=1,2,4,5$. All these trees satisfy the neighbour sum property with $f(v_1)=f(v_2)=1$, $f(v_4)=f(v_5)=-1$ and $f(v)=0$ for all other vertices (see Figure \ref{paths}). So, using Theorem \ref{otter}, we have
    $$\sigma(n)\ge \tau(n-4) \sim C\cdot \alpha^{n-4} \cdot (n-4)^{-5/2}$$
    and hence for large $n$, we have
    $$\frac{\sigma(n)}{\tau(n)}\ge \frac 1{\alpha^4} \cdot \left(\frac{n}{n-4}\right)^{5/2} \xrightarrow[]{n\to \infty} \frac 1{\alpha^4}\approx 0.0131$$
    hence completing the proof.
\end{proof}

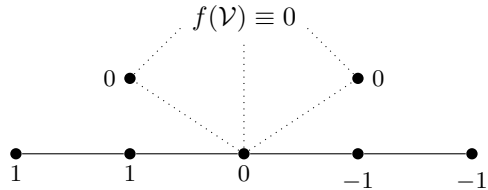
\begin{figure}[htp!]
    \centering
    \begin{tikzpicture}[
    every node/.style={circle, fill=black, inner sep=1.5pt},
    every label/.style={fill=none, inner sep=1pt, font=\footnotesize},
    expr/.style={draw=none, fill=none, shape=rectangle, font=\small}
    ]
    
    \node[label=below:{$1$}] (v0) at (0,0) {};
    \node[label=below:{$1$}] (v1) at (1.5,0) {};
    \node[label=below:{$0$}] (v2) at (3,0) {};
    \node[label=below:{$-1$}] (v3) at (4.5,0) {};
    \node[label=below:{$-1$}] (v4) at (6,0) {};
    \node[label=left:{$0$}] (v5) at (1.5,1) {};
    \node[label=right:{$0$}] (v6) at (4.5,1) {};
    
    \draw (v0) -- (v1) -- (v2) -- (v3) -- (v4);
    \draw[dotted] (v2) -- (v5);
    \draw[dotted] (v2) -- (v6);
    \draw[dotted] (v2) -- ++(0,1.5);
    \draw[dotted] (v5) -- ++(0.7,0.7);
    \draw[dotted] (v6) -- ++(-0.7,0.7);
    \node[expr] at ($(v2)+(0,1.8)$) {$f(\mathcal{V})\equiv 0$};
    
    \end{tikzpicture}
    \caption{Schematic diagram for the neighbour sum property satisfied on trees of arbitrary sizes ($\geq 5$ vertices)}
    \label{paths}
\end{figure}

The following is the output of the code to generate the number of non-isomorphic trees on $n$ vertices that satisfy the neighbour sum property, where $n$ runs from $1$ to $20$ (see code in Section \ref{app}):
$$0,1,0,0,1,2,2,6,14,29,63,166,405,977,2481,6530,16757,43534, 115700,308527$$

\section{Infinite Trees}\label{infinite}
So far, we have only considered the problem on finite trees. It should be noted that the finiteness of $\mathcal T$ is crucial in the proof of Theorem \ref{main} since it uses the values at all the leaves of the tree as well as a notion of the largest path starting from $v_0$. We shall therefore look at infinite trees with a fresh perspective.

When looking at infinite trees, we are interested in locally finite trees, so that we can continue with our current definition of the neighbour sum property. We will first prove that infinite trees which \textit{does not have too many leaves} always satisfy the neighbour sum property. The techniques we will use are similar to those used in Section $6$ of \cite{sayan}). The idea is to explicitly construct a solution using a breadth-first search approach.

\begin{prp}
    Let $\mathcal T = (\mathcal V, \mathcal E)$ be a locally finite infinite tree such that there is a $v_0\in \mathcal V$ such that for any $v\in A_j$, the children of $v$ are not all leaves. Then, $\mathcal T$ satisfies the neighbour sum property,
\end{prp}
\begin{proof}
    For a vertex $v\in \mathcal V$, let $C(v)$ be its children and let $L(v)$ and $N(v)$ be those that are leaves and not leaves respectively. By the given assumption, $\left\lvert N(v)\right\rvert \neq 0$ for all $v\in \mathcal V$.

    Now, begin constructing the neighbour sum function by setting $f(v_0)=1$ and proceed inductively from there. The neighbour sum property for $v_0$ gives
    $$\sum_{u\in N(v_0)} f(u) = \left(1-\left\lvert L(v_0)\right\rvert\right)\cdot f(v_0)$$
    and hence we only need to assign values to the vertices in $N(v_0)$ such that they sum up to $\left(1-\left\lvert L(v_0)\right\rvert\right)\cdot f(v_0)$. This can be done since $\left\lvert N(v_0)\right\rvert \neq 0$.

    Now, for a vertex $v$ with a parent $p$ (which already has a value under $f$ by induction), the neighbour sum property gives
    $$\sum_{u\in N(v)} f(u) = \left(1-\left\lvert L(v)\right\rvert\right)\cdot f(v) - f(p)$$
    and hence, now we need to assign values to the vertices in $N(v)$ such that they sum up to $\left(1-\left\lvert L(v)\right\rvert\right)\cdot f(v) - f(p)$.

    This function is clearly non-zero since $f(v_0)=1$.
\end{proof}

We will now show that the extra assumption in the previous proposition cannot be dropped.
\begin{thm}
    Not all infinite trees satisfy the neighbour sum problem.
\end{thm}
\begin{proof}
    We will prove this by explicitly constructing an example.
    
    Consider the integers indexed by $\{v_n\}$ for each $n\in \mathbb Z$. Now, consider two other vertices $u_n$ and $w_n$ laid down to form a path with $v_n$. So, the edges of our graph are $\{v_n,v_{n+1}\}_{n\in \mathbb Z}$, $\{v_n,u_n\}_{n\in \mathbb Z}$ and $\{u_n,w_n\}_{n\in \mathbb Z}$ (see Figure \ref{Trinf}). We claim that this graph does not have any solutions for the neighbour sum problem. On the contrary, let us assume that $f$ is a solution.

    The equations for $u_n$ and $v_n$ are
    \begin{align*}
        f(u_n)&=f(w_n)+f(v_n)\\
        &=f(u_n)+f(v_n)
    \end{align*}
    since $f(w_n)=f(u_n)$ as $w_n$ is a leaf. This gives $f(v_n)=0$ for all $n$.

    But the equation for $v_n$ is
    $$f(v_n)=f(v_{n-1})+f(v_{n+1})+f(u_n)$$
    hence proving $f(u_n)=f(w_n)=0$.

    This completes the proof.
\end{proof}

\begin{figure}[htp!]
    \centering
    \begin{tikzpicture}[
    vertex/.style={circle, draw=black, fill=black, minimum size=4pt, inner sep=0pt},
    dotlabel/.style={font=\footnotesize}
    ]
    
    \foreach \i in {-2,-1,0,1,2} {
        \node[vertex, label=below:{\footnotesize$v_{\i}$}] (n\i) at (1.5*\i, 0) {};
    }
    
    \foreach \i [evaluate=\i as \j using int(\i+1)] in {-2,-1,0,1} {
        \draw (n\i) -- (n\j);
    }
    
    \foreach \i in {-2,-1,0,1} {
        \node[vertex, label=right:{$u_{\i}$}] (u\i) at ([shift=(90:1cm)] n\i.center) {};
        \node[vertex, label=right:{$w_{\i}$}] (v\i) at ([shift=(90:2cm)] n\i.center) {};
        \draw (n\i) -- (u\i)
              (u\i) -- (v\i);
    }
    \node[vertex, label=left:{$u_{2}$}] (u2) at ([shift=(90:1cm)] n2.center) {};
    \node[vertex, label=left:{$w_{2}$}] (v2) at ([shift=(90:2cm)] n2.center) {};
    \draw (n2) -- (u2)
          (u2) -- (v2);

    \draw[dashed, thick] (n-2) to[out=180, in=180] ++(-0.8,0) node[left] {\scriptsize $\cdots$};
    \draw[dashed] (n2) to[out=0, in=0] ++(0.8,0) node[right] {\scriptsize $\cdots$};
    
    \end{tikzpicture}
    \caption{An infinite tree that doesn't satisfy the neighbour sum property.}
    \label{Trinf}
\end{figure}
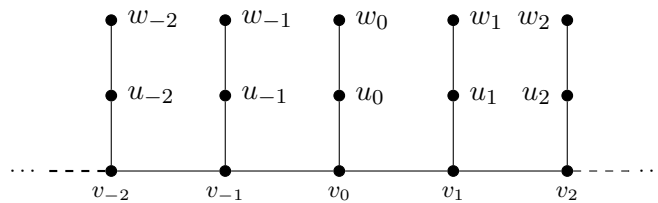

\section{Concluding Remarks}\label{rem}
Of course, our ultimate goal is to know which of the more general graphs solve this problem (see Question 31 in \cite{sayan} for the formal statement). While we do not have any concrete insights to offer on that, we would like to report the following observations that might be useful.
\begin{prp}\label{cycle}
    Cycle graphs satisfy the neighbour sum property if and only if the cardinality of the vertex set is a multiple of $6$.
\end{prp}
\begin{proof}
    Let $\{v_1,\dots , v_n\}$ be the vertices of an $n$-cycle. Then, a possible neighbour sum function $f$ must satisfy
    \begin{align*}
        -f(v_{\overline{i}})
        &= f(v_{\overline{i+2}}) - f(v_{\overline{i+1}})\\
        &= f(v_{\overline{i+1}}) + f(v_{\overline{i+3}}) - f(v_{\overline{i+1}})\\
        &= f(v_{\overline{i+3}})
    \end{align*}
    where $\overline j$ is used as a shorthand for $j(\text{mod } n)$.

    This prove that such a function must have a period of $6$ and hence for $f(v_1)$ to equal $f(v_{\overline{n+1}})$, $n$ has to be a multiple of $6$.

    If $n$ is indeed a multiple of $6$, setting $$f(v_k)=\begin{cases}
        1\;\;&k\equiv  0,1\;(\text{mod } 6)\\
        0\;\;&k\equiv 2,5\;(\text{mod } 6)\\
        -1\;\;&k\equiv 3,4\;(\text{mod } 6)
        
    \end{cases}$$ satisfies the neighbour sum property.\footnote{It is straightforward to verify that $f(v_{\overline 1})=-f(v_{\overline 4})=1$, $f(v_{\overline 2})=-f(v_{\overline 5})=3$ and $f(v_{\overline 3})=-f(v_{\overline 6})=2$ (where $\overline j \coloneqq j(\text{mod } 6)$) gives another class of solutions.} 
\end{proof}

Interpreting cycle graphs as one-dimensional toroidal chessboards, one can also use the techniques demonstrated in Section $4$ of \cite{sayan} to prove Proposition \ref{cycle}.

\begin{lem}\label{bip}
    The complete bipartite graph $K_{m,n}$ satisfies the neighbour sum property if and only if $m=n=1$.
\end{lem}
\begin{proof}
    Let $P$ and $Q$ be the sets having $m$ and $n$ vertices respectively, and $f$ be a real valued non-zero function on its vertex set satisfying the neighbour sum property. Let elements of $P$ be labelled $x_1,\:x_2,\ldots,\:x_m$ and elements of $Q$ be labelled $y_1,\:y_2,\ldots,\:y_n$ under $f$. Since each element of $P$ in connected to every element of $Q$ (and vice versa), we will get
    $$x_1=\sum_{i=1}^n y_i=x_2=\ldots=x_m$$
    and
    $$y_1=\sum_{j=1}^mx_j=y_2=\ldots=y_n$$
    thus giving the equations $x_1=ny_1$ and $y_1=mx_1$, which have a unique non-zero solution if and only if $m=n=1$.
\end{proof}

\begin{lem}\label{complete}
    The complete graph $K_n$ satisfies the neighbour sum property if and only if $n=2$.
\end{lem}
\begin{proof}
    Under $f$, label the values taken by the vertices in any order as $x_1,\;x_2,\ldots,x_n$. Define $S=\sum_{i=1}^n x_i$.

    Now check that for any $1\leq i,j\leq n$, with $i\neq j$, $x_i=x_j+(S-x_j-x_i)$ and $x_j = x_i+(S-x_j-x_i)$. This gives us $x_i=x_j$, which means that the values taken by all vertices need to be the same, say $x$. Since each vertex is connected once to $n-1$ other vertices, one must have $x=(n-1)x$. This has a non-zero solution if and only if $n=2$.
\end{proof}

\begin{prp}
    The complete $k$-partite graph $K_{n_1,n_2,\ldots,n_k}$ satisfies the neighbour sum property if and only if it is $K_2 \cong K_{1,1}$.
\end{prp}
\begin{proof}
Due to being complete, each partition will have all vertices bear the same value under any $f$ under similar arguments as in Lemma \ref{bip}. Label the values for the partitions as $x_1,\:x_2,\ldots,\:x_k$. Now note that for any $1\leq i,j\leq k$ with $i\neq j$, we have
$$x_i=n_jx_j+\sum_{r\neq i,j}n_rx_r,\;\;x_j=n_ix_i+\sum_{s\neq i,j}n_sx_s$$
which again has a unique solution if $(1+n_i)x_i=(1+n_j)x_j$. For fixed non-zero $n_i$'s, all values should have the same sign (note that $f\not\equiv 0$). Let us suppose all $x_i$s are positive. Then $$n_jx_j = \dfrac{n_j}{1+n_j}(1+n_i)x_i\geq x_i$$ with equality when $n_i=n_j=1$. Since this is true for any $i,j$, the complete $k$-partite can only have a solution if all the partitions have one vertex each, and the result follows from an application of Lemma \ref{complete}.
\end{proof}

\bibliographystyle{plain}
\bibliography{Chotushkone}

\section{Acknowledgements}
We would like to thank Satvik Saha and Ayanava Mandal for their meticulous scrutiny of the manuscript and helpful suggestions.

\clearpage
\section{Appendix}\label{app}
We present the code to compute the number of non-isomorphic trees on $n$ vertices that satisfy the neighbour sum property below.

\begin{minted}{python}
    import networkx as nx
    import math
    from networkx.generators.nonisomorphic_trees import nonisomorphic_trees
    
    def safe_inv(x):
        """Handles formal rules of algebra: 1/0 = inf, r/inf = 0"""
        if x == 0:
            return float('inf')
        elif math.isinf(x):
            return 0.0
        else:
            return 1.0 / x
    
    def compute_S(tree, node, parent):
        """Recursively compute S_{jk} from a node in a rooted tree."""
        children = [n for n in tree.neighbors(node) if n != parent]
        if not children:
            return 0.0
        subS = [compute_S(tree, child, node) for child in children]
        return sum([safe_inv(1 - s) for s in subS])
    
    def tree_satisfies_ns(tree):
        """Check if the tree satisfies the NS property by checking S_{1k}(v) = 1."""
        for root in tree.nodes:
            if math.isclose(compute_S(tree, root, None), 1.0, rel_tol=1e-9):
                return True
        return False
    
    def count_ns_trees(max_n):
        """Return a dictionary of {n: number of trees satisfying NS} for n = 1 to max_n."""
        ns_counts = {}
        ns_counts[1]=0 #Assumption
        ns_counts[2]=1 #This is the path graph on two vertices
        print(ns_counts[1],'\n',ns_counts[2])
        for n in range(3, max_n + 1):
            count = 0
            for tree in nonisomorphic_trees(n):
                if tree_satisfies_ns(tree):
                    count += 1
            ns_counts[n] = count
            print(ns_counts[n])
        return ns_counts
    ns_trees = count_ns_trees(20)
\end{minted}
\end{document}